\documentclass[12pt, reqno]{amsart}
\usepackage{amsmath, amsthm, amscd, amsfonts, amssymb, graphicx, color}
\usepackage[cp1251]{inputenc}
\usepackage[bookmarksnumbered, colorlinks, plainpages]{hyperref}
\hypersetup{colorlinks=true,linkcolor=red, anchorcolor=green, citecolor=cyan, urlcolor=red, filecolor=magenta, pdftoolbar=true}
\newtheorem{theorem}{Theorem}[section]

\theoremstyle{definition}

\theoremstyle{remark}
\newtheorem{remark}[theorem]{Remark}
\numberwithin{equation}{section}

\begin{document}

\setcounter{page}{1}

\title[Nonlocal in time reaction-diffusion equations]{Global behavior of nonlocal in time reaction-diffusion equations}

\author[B. T. Torebek]{Berikbol T. Torebek}

\address{\textcolor[rgb]{0.00,0.00,0.84}{Berikbol T. Torebek \newline Institute of
Mathematics and Mathematical Modeling \newline 125 Pushkin str.,
050010 Almaty, Kazakhstan \newline Department of Mathematics: Analysis, Logic and Discrete Mathematics \newline Ghent University \newline Krijgslaan 281, Building S8, B 9000 Ghent, Belgium}}
\email{\textcolor[rgb]{0.00,0.00,0.84}{berikbol.torebek@ugent.be}}


\let\thefootnote\relax\footnote{$^{*}$Corresponding author}

\subjclass[2010]{Primary 35K55, 35R11; Secondary 35K57, 35A01.}

\keywords{Sonine kernel, reaction-diffusion equation, global existence, decay estimate.}

\begin{abstract} The present paper considers the Cauchy-Dirichlet problem for the time-nonlocal reaction-diffusion equation $$\partial_t (k\ast(u-u_0))+\mathcal{L}_x [u]=f(u),\,\,\,\, x\in\Omega\subset\mathbb{R}^n, t>0,$$ where $k\in L^1_{loc}(\mathbb{R}_+),$ $f$ is a locally Lipschitz function, $\mathcal{L}_x$ is a linear operator. This model arises when studying the processes of anomalous and ultraslow diffusions. Results regarding the local and global existence, decay estimates, and blow-up of solutions are obtained. The obtained results provide partial answers to some open questions posed by Gal and Varma \cite{Gal}, as well as Luchko and Yamamoto \cite{Luchko1}.
Furthermore, possible quasi-linear extensions of the obtained results are discussed, and some open questions are presented.
\end{abstract}
\maketitle
\tableofcontents
\section{Introduction}

In this paper, we consider a nonlocal in time reaction-diffusion equation
\begin{equation}  \label{1}\partial_t (k\ast(u-u_0))+\mathcal{L}_x [u]=f(u),\,\,\,\, x\in\Omega\subset\mathbb{R}^n, t>0,\end{equation}
with the inhomogeneous Cauchy data
\begin{equation}\label{2}u(x,0)=u_0(x),\,\,\,\,\,x\in\Omega,\end{equation}
and with the Dirichlet boundary condition
\begin{equation}  \label{3}u(x,t)=0,\,\,\,\, x\in\mathbb{R}^n\setminus\Omega,\,t\geq 0,\end{equation}
where $k\in L^1_{loc}(\mathbb{R}_+),$ $u_0$ is a given function, $f$ is a locally Lipschitz function, $\mathcal{L}_x$ is a linear operator. Here $\ast$ denotes the Laplace convolution, i.e.,
$$(k\ast v)(t)=\int\limits_0^tk(t-s)v(s)ds.$$

To study the problem \eqref{1}-\eqref{3}, we need to assume the following hypotheses:
\begin{description}
\item[(A)] The kernel function $k\in L^1_{loc}(\mathbb{R}_+)$ is nonnegative and nonincreasing, and there exists a function $l\in L^1_{loc}(\mathbb{R}_+),$ such that $(k\ast l)(t)=1$ for $t>0$.
\item[(B)] Suppose that $\mathcal{L}_x: D(\mathcal{L})\rightarrow L^2(\Omega)$ is an unbounded linear operator with boundary condition \eqref{3}, and $\mathcal{L}_x$ is the self-adjoint operator with positive discrete spectrum. There exist $C_\mathcal{L}>0$ such that if $u$ is as in \eqref{1}-\eqref{3}, then
\begin{equation}\label{Lineq}\int\limits_\Omega u(x,t)\mathcal{L}_x[u](x,t)dx\geq C_\mathcal{L}\|u(\cdot,t)\|^2_{L^2(\Omega)},\,t>0.\end{equation}
\item[(C)] The source function $f$ is locally Lipschitz and satisfies
\begin{align*}&f(y)=0,\,\,\,\,\,\text{if}\,\,\,y=0\,\,\text{or}\,\,y=1,\\&
f(y)< 0,\,\,\,\,\,\text{if}\,\,\,y\in(0,1),\\&
f(y)> 0,\,\,\,\,\,\text{if}\,\,\,y>1\,\,\text{or}\,\,y<0.
\end{align*}
\end{description}

The equation \eqref{1} describes many processes in physics; we will provide further details in subsection \ref{ApplPhys}. Also, one can find more information about the history of problem \eqref{1} in subsection \ref{HisBack}.

The main purpose of this paper is to study the qualitative properties of solutions to problem \eqref{1}-\eqref{3}, and the following will be the main results:
\begin{itemize}
\item Existence of local in time unique strong solution and its behavior for initial data satisfying $0\leq u_0\leq 1$;
\item Existence of global in time strong solutions, which satisfies
$$0\leq u(x,t)\leq 1,\, (x,t)\in \bar\Omega\times[0,\infty)$$ for $0\leq u_0\leq 1$;
\item Decay estimates of global in time solutions for $0\leq u_0\leq 1$;
\item Finite time blow-up solutions for $\int\limits_\Omega u_0(x)\phi_1(x)dx\geq M,\, M>0,$ where $\phi_1>0$ is a first eigenfunction of operator $\mathcal{L}_x$.
\end{itemize}
It should be noted that the results obtained in this paper will partially answer some open questions (more details can be found in Remarks \ref{rem2} and \ref{rem3} in Section \ref{Main}) posed by Gal and Warma in \cite{Gal} and posed by Luchko and Yamamoto in \cite{Luchko1}.

\subsection{Examples of (A), (B) and (C)}\label{exABC}
In this subsection, we demonstrate some examples showing that the class of hypotheses (A), (B) and (C) is not empty.
\subsubsection{Examples of (A)} Condition (A) was first considered by Sonine in \cite{Sonine}, since then a pair belonging to the class (A) is called the Sonine kernel, and (A) is called the Sonine condition. The simplest example of condition (A) is the pair $k(t)=\delta(t),\, l(t)=1$, which generates $\frac{d}{dt},\,\int\limits_0^t$. Here $\delta$ is a Dirac delta function.

We provide further examples that demonstrate that the Sonine condition (A) encompasses a wide class of different integro-differential operators.
\begin{itemize}
\item The most famous example of a pair $(k,l)$ is $$k(t)=g_{1-\alpha}(t),\,l(t)=g_{\alpha}(t),\,0<\alpha<1,$$ where $g_{\mu}(t)=\frac{t^{\mu-1}}{\Gamma(\mu)},\,\mu>0$. Here $\Gamma(\cdot)$ is the Euler's gamma function.

    In this case $\partial_t (k\ast v)(t)$ coincides with the Riemann-Liouville time-fractional derivative (see \cite{Kilbas}) $$\partial_t (g_{1-\alpha}\ast v)(t)=\frac{1}{\Gamma(1-\alpha)}\partial_t\int\limits_0^t(t-s)^{-\alpha}v(s,x)ds$$ of order $\alpha\in(0,1)$.
\item Let $(k,l)$ be defined as follows $$k(t)=\int\limits_0^1g_\alpha(t)d\alpha,\, l(t)=\int\limits_0^\infty \frac{e^{-st}}{1+s}ds,\,t>0.$$
    In this case $\partial_t (k\ast v)(t)$ coincides with the distributed order derivative \cite{Kochubei1}.
\item Let $$k(t)=g_{1-\alpha}(t)e^{-\mu t},\, l(t)=g_\alpha(t)e^{-\mu t}+\mu\int\limits_0^tg_\alpha(s)e^{-\mu s}ds,\,t>0,$$ then $\partial_t (k\ast v)(t)$ coincides with the tempered fractional derivative \cite{Zacher1}.
\item Another well-known example of pair $(k,l)$ is (see \cite{Sonine}) $$k(t)=(\sqrt{t})^{\alpha-1}J_{\alpha-1}\left(2\sqrt{t}\right),\,\,l(t)=(\sqrt{t})^{-\alpha}I_{-\alpha}\left(2\sqrt{t}\right),\,\alpha\in(0,1),$$ where $$J_{\nu}(y)=\sum\limits_{k=0}^\infty\frac{(-1)^k(y/2)^{2k+\nu}}{k!\Gamma(k+\nu+1)}, \,I_{\nu}(y)=\sum\limits_{k=0}^\infty\frac{(y/2)^{2k+\nu}}{k!\Gamma(k+\nu+1)}$$ are the Bessel and the modified Bessel functions, respectively.
\item Our last example of pair $(k,l)$ is the following (see \cite{Luchko2}) $$k(t)=g_{1-\beta+\alpha}(t)+g_{1-\beta}(t),\,\,l(t)=t^{\beta-1}E_{\alpha,\beta}(-t^\alpha),\,0<\alpha<\beta<1,$$ where $$E_{\alpha,\beta}(y)=\sum\limits_{k=0}^\infty\frac{(-1)^ky^{k}}{\Gamma(\alpha k+\beta)}$$ is the Mittag-Leffler function.
\end{itemize}
\subsubsection{Examples of (B)}
Note that different operators satisfying condition (B) were previously studied in \cite{Vald1} and \cite{Vald2}. For the convenience of readers, below we give several examples of operators that satisfy (B).
\begin{itemize}
\item Laplace operator $-\Delta=-\sum\limits_{j=1}^n\frac{\partial^2}{\partial x_j^2}.$ Laplace operator with Dirichlet conditions satisfies the classical Poincar\'{e} inequality $$(-\Delta v, v)=\|\nabla v\|^2_{L^2(\Omega)}\geq \lambda_1\|v\|^2_{L^2(\Omega)},$$ which is a particular case of inequality \eqref{Lineq}. Here $\lambda_1>0$ is a first eigenvalue of Laplace-Dirichlet operator.
\item Fractional Laplace \begin{equation*}(-\Delta)^sv(x)=C_{n,s}\,\text{P.V.}\int_{\mathbb{R}^n}\frac{v(x)-v(y)}{|x-y|^{n+2s}}dy,
\end{equation*} where \begin{equation*}C_{n,s}=\frac{s2^{2s}}{\pi^{\frac{n-1}{2}}}\frac{\Gamma(\frac{n+2s}{2})}{\Gamma(\frac{2+1}{2})\Gamma(1-s)}\end{equation*} is a normalization constant and “P.V.” is an abbreviation for “in the principal value  sense”.

For the fractional Laplace, the fractional Poincar\'{e} inequality holds \cite{DiN}.
\item Differential operator with involution $$\mathcal{L}[v]=-v''(x)+\varepsilon v''(1-x),\,x\in(0,1),$$ where $|\varepsilon|<1.$ Multiplying it by $v$ and integrating by parts, we have
    \begin{align*}\int\limits_0^1\mathcal{L}[v]vdx=&\int\limits_0^1|v'(x)|^2dx+\varepsilon\int\limits_0^1v'(x) v'(1-x)dx\\&\geq \int\limits_0^1|v'(x)|^2dx-|\varepsilon|\left|\int\limits_0^1v'(x) v'(1-x)dx\right|\\& \geq \|v'\|^2_{L^2(0,1)}-|\varepsilon|\|v'\|_{L^2(0,1)}\left(\int\limits_0^1 |v'(1-x)|^2dx\right)^{\frac{1}{2}}\\&=(1-|\varepsilon|)\|v'\|^2_{L^2(0,1)}\geq  (1-|\varepsilon|)\pi^2\|v\|^2_{L^2(0,1)},\end{align*} thanks to the H\"{o}lder and Poincar\'{e} inequalities.
\item Fractional Sturm-Liouville operator $$\mathcal{L}[v]=D^{\alpha}_{b-}\partial^\alpha_{a+}v,\,\alpha\in(0,1),$$ where $$D^{\alpha}_{b-}v=-\partial_x\int\limits_x^b\frac{(\xi-x)^{-\alpha}}{\Gamma(1-\alpha)}v(\xi)d\xi\,\,\,\text{and}\,\,\, \partial^\alpha_{a+}v=\int\limits_a^x\frac{(x-\xi)^{-\alpha}}{\Gamma(1-\alpha)}v'(\xi)d\xi$$ are the right Riemann-Liouville and the left Caputo space fractional derivatives, respectively.

    Similarly to the previous one, we have
    \begin{align*}\int\limits_a^b\mathcal{L}[v]vdx&=\int\limits_a^b|\partial^\alpha_{a+}v(x)|^2dx\geq C \|v\|^2_{L^2(a,b)},\,\,C>0,\end{align*}
    thanks to the fractional integration by parts formula and fractional Sobolev-Poincar\'{e} inequality (see \cite{Torebek}).
\end{itemize}
\subsubsection{Examples of (C)}
\begin{itemize}
\item Fisher-KPP-type nonlinearity $$f(v)=v(v-1),$$ which is reduced to the Fisher-KPP nonlinearity $f(w)=w(1-w)$ using the transformation $w\hookrightarrow 1-v.$ Also one can consider a more general function $$f(v)=|v|^{q-1}v(|v|^{p-1}v-1),$$ where $p,q>1.$
\item Logarithmic nonlinearity $$f(v)=v(v-1)\log(1+|v|);$$
\item Exponential nonlinearities $$f(v)=(e^{v}-1)(e^{v}-e),$$ and $$f(v)=v(e^{v-1}-1),$$
\item Hyperbolic nonlinearities $$f(v)=(v-1)\sinh v,$$ and $$f(v)=v\tanh(v-1).$$
\end{itemize}
\subsection{Applications in physics}\label{ApplPhys} The main motivation for studying nonlocal in time reaction-diffusion equations \eqref{1}, comes from physics.

In particular, if $$k(t)=g_{1-\alpha}(t)=\frac{t^{-\alpha}}{\Gamma(1-\alpha)},\,\,\alpha\in(0,1),$$
then the model equation \eqref{1} with the time-fractional derivative describes the process of anomalous diffusion, dynamic processes in materials with memory, diffusion in fluids in porous media with memory et al. More details can be found in a number of sources (see for example \cite{Diet, Hilfer, Main, Metz1, Prus, Tarasov1, Tarasov2}).

For the convenience of readers, below we give an example of describing the process of anomalous diffusion using model \eqref{1}. Consider the following toy model of the influence of time and space derivatives on the diffusion process
$$\partial_t(g_{1-\alpha}\ast (h-h_0))(t)+(-\Delta)^sh=0,$$ where $0<\alpha<1,$ $h=h(x,t),$ $h_0=h(x,0)$ and $(-\Delta)^s,\,0<s<1$ is the fractional power of Laplace operator. The mean squared displacement corresponding to this equation is given by
$$\langle x^{2s}\rangle\sim t^{\alpha}.$$
As observed, the mean squared displacement demonstrates power law behavior. The sub-diffusive processes arise from $\alpha<s,$ while the super-diffusive processes arise from $\alpha+>s.$ The normal diffusive processes arise from $\alpha=s=1.$

Let us take another example, let $$k(t)=\int\limits_0^1g_\alpha(t)d\alpha,$$ that is, consider the equation \eqref{1} with distributed order time-derivative. The mean squared displacement corresponding to this equation will have a logarithmic growth, and in this case \eqref{1} describes ultraslow diffusion process (see \cite{Kochubei1, Luch1}).

\subsection{Historical background}\label{HisBack} In recent years, the equation \eqref{1} with an integro-differential operator $\partial_t (k\ast \cdot)$ in time has been studied more often. Obviously, on the one hand, this is due to the numerous physical applications of equation \eqref{1}, as we have presented in previous subsection \ref{ApplPhys}. On the other hand, this is related from a mathematical point of view, because, as we have shown in subsection \ref{exABC}, operator $\partial_t (k\ast \cdot)$ covers a wide class of previously known or unknown integro-differential operators.

There are numerous works (see for example \cite{Gal,Ke,Kochubei1,Kochubei,Luch1,Luchko1,Prus,Zacher,Zacher1,Zacher2}) devoted to the existence and uniqueness of the solution of problems of type \eqref{1}. Many of them were devoted to linear equations (that is, when $f(u)$ is a linear function in $u$), while others studied equation \eqref{1} on a finite time interval $(0,T)$, or for special cases of $\partial_t (k\ast \cdot)$.

When $k(t)=g_{1-\alpha}(t)=\frac{t^{-\alpha}}{\Gamma(1-\alpha)},\,\,\alpha\in(0,1),$ the problem \eqref{1} is well studied by many authors. The problems closer to \eqref{1} were studied in \cite{AAAKT15, Gal}, where $f$ had Fisher-KPP type non-linearities. For example, in \cite{AAAKT15} was studied the equation $$\partial_t (g_{1-\alpha}\ast(u-u_0))(t)-\Delta u=u(u-1),\,t>0,\,x\in\Omega,$$ and the results of local and global existence, and finite time blow-up solutions were obtained. Gal and Warma in \cite{Gal} considered the abstract equation $$\partial_t (g_{1-\alpha}\ast(u-u_0))(t)+Au=f(u),$$ where $A$ is a closed linear operator with domain $D(A)$ in a Banach space $X,$ and $f(u),\,\,f(0)=f(1)=0$ is a locally Lipschitz function. The existence of a strong solution has been proven, and several \textbf{open questions} related to the properties of the solutions have been proposed; in particular, it was proposed to prove finite time blow-up of solutions.

Below we briefly dwell on some works, where was studied the equations of the form \eqref{1}, for general cases of $\partial_t (k\ast \cdot)$.

Kochubei in \cite{Kochubei} considered the homogeneous equation $$\partial_t (k\ast(u-u_0))-\Delta u=0,\,x\in \mathbb{R}^n$$ and studied some properties of its fundamental solutions. In \cite{Zacher1}, Vergara and Zacher considered the above problem in the bounded domain $\Omega\subset\mathbb{R}^n,$ and proved the optimal decay estimates of the weak solutions.

In \cite{Zacher2}, Vergara and Zacher studied equation
$$\partial_t (k\ast(u-u_0))-\text{div}(A(t,x)\nabla u)=f(u),\,x\in \Omega,$$ with Cauchy-Dirichlet conditions \eqref{2}-\eqref{3},
and obtained a number of results, including stability, instability, and blow-up of the weak solution. Here $A(t,x)\in L^\infty((0,T)\times\Omega)$ and $f(u)$ is a locally Lipschitz function.

Luchko and Yamamoto \cite{Luchko1} proved the maximum principle for equation $$\partial_t (k\ast(u-u_0))(t)=\sum\limits_{i,j=1}^na_{i,j}(x)\frac{\partial^2u}{\partial x_i\partial x_j}+\sum\limits_{i=1}^nb_{i}(x)\frac{\partial u}{\partial x_i}u-q(x)u+f(t,x),\,(0,T]\times\Omega,$$ where $a_{i,j}(x)=a_{j,i}(x),$ $b_i(x),$ $q(x)$ and $f(t,x)$ are continuous functions. Also in \cite{Luchko1} an \textbf{open question} was discussed regarding the asymptotic behavior of the solution for large time.

In \cite{Ke}, Ke et al. considered problem $$\partial_t (k\ast(u-u_0))+\mathcal{L}u=f(u),$$ for more general self-adjoint abstract unbounded operator $\mathcal{L},$ and proved the existence of strong solution on the time interval $(0,T),\,\,T<\infty.$

All the above works motivated us to consider problem \eqref{1},\eqref{2},\eqref{3}, and we will try to partially answer some open questions posed in \cite{Gal} and \cite{Luchko1}.

\section{Main results}\label{Main}
In this section, we present the main results and some particular cases of them.

Let $V_s(\Omega)=D(\mathcal{L}^s),$ where $$D(\mathcal{L}^s)=\left\{v=\sum\limits_{k=1}^\infty v_ke_k:\,\,\,\sum\limits_{k=1}^\infty\lambda_k^{2s}v_k^2<\infty\right\}.$$ Here $\left\{e_k\right\}_{k=1}^\infty$ is the system of orthonormal eigenfunctions of $\mathcal{L},$ and $\lambda_k>0,\,k=1,2,...$ is the eigenvalue corresponding to the $e_k.$
Then the space $V_s(\Omega)$ equipped with the norm
$$\|v\|_{V_s(\Omega)}=\|\mathcal{L}^sv\|_{L^2(\Omega)}=\left(\sum\limits_{k=1}^\infty\lambda_k^{2s}v_k^2\right)^{\frac{1}{2}}$$
is the Banach space (see \cite{Ke}).

Below we formulate a statement about the existence of local in time solutions.
\begin{theorem}\label{Th2} Let the hypotheses (A), (B) and (C) be satisfied. Assume that $u_0\in V_{1/2}(\Omega),$ then there exists $T^*>0$ such that the problem \eqref{1},\eqref{2},\eqref{3} has a unique strong solution $$u\in C([0,T^*], V_{1/2}(\Omega))\cap C^\gamma([\delta,T^*], L^2(\Omega)),\,\, \gamma\in(0,1), 0<\delta<T^*.$$ In addition, if $0\leq u_0(x)\leq 1,\,x\in\bar\Omega,$ then $$0\leq u(x,t)\leq 1$$ for $(x,t)\in \bar\Omega\times[0,T^*].$
\end{theorem}

\begin{remark}
Some of the results of Theorem \ref{Th2} are direct consequences of the results from \cite{Ke}. However, since we have special $f(u)$ from class (C), the second part of Theorem \ref{Th2} is new and important for obtaining further results regarding global solutions.
\end{remark}

\begin{theorem}\label{Th1} Let hypotheses (A), (B) and (C) hold. Suppose that $u_0\in V_{1/2}(\Omega).$
\begin{description}
\item[(i)]
Let $0\leq u_0(x)\leq 1,\,\,u_0\not\equiv 0,\,\,x\in\bar\Omega.$ Then the problem \eqref{1} with the initial-boundary conditions \eqref{2}-\eqref{3} admits a global strong solution $$u\in C([0,\infty),V_{1/2}(\Omega))\cap C^\gamma([\delta,\infty), L^2(\Omega)), \,\gamma\in(0,1), 0<\delta<T<+\infty,$$ which satisfies
$$0\leq u(x,t)\leq 1$$ for $(x,t)\in \bar\Omega\times[0,\infty);$
\item[(ii)] Assume that $0\leq u_0(x)\leq 1,\,\,u_0\not\equiv 0,\,\,x\in\bar\Omega$ and let $\|u_0\|_{L^2(\Omega)}\neq 0.$ Then the global strong solution $0 \leq u\leq 1$ of problem \eqref{1} supplemented with the initial-boundary conditions \eqref{2} and \eqref{3} satisfies the estimate \begin{equation}\label{7*}\|u(t,\cdot)\|_{L^2(\Omega)}\leq \frac{1}{1+C_{\mathcal{L}}(1\ast l)(t)}\|u_0\|_{L^2(\Omega)},\,\,\,t\geq 0,\end{equation} where $C_\mathcal{L}$ is the positive constants of inequality \eqref{Lineq} in (B).

    In addition, if $l\not\in L^1(\mathbb{R}),$ then
    \begin{equation*}\|u(t,\cdot)\|_{L^2(\Omega)}\rightarrow 0,\,\,\text{at}\,\,\,t\rightarrow+\infty.\end{equation*}
\item[(iii)] Let $\lambda_1>0$ is a first eigenvalue and $\phi_1(x)>0$ is a corresponding
 eigenfunction of operator $\mathcal{L}$ with Dirichlet boundary conditions and let $f: \mathbb{R}_+ \rightarrow \mathbb{R}$ is a convex function. Assume that there exists $m > 1$ such that
    \begin{equation}\label{7**}\int\limits_m^\infty\frac{dy}{f(y)}<\infty.\end{equation}
Then, there exists $M = M(\lambda_1, f, l) > 0$ such that for any $u_0\geq 0,\,u_0\not\equiv 0$ satisfying \begin{equation}\label{7***}\int\limits_\Omega u_0(x)\phi_1(x)dx\geq M,\end{equation} the strong solution $u$ of problem \eqref{1}-\eqref{3} blows up in finite time.
\end{description}
\end{theorem}
\begin{remark}
Part (i) of Theorem \ref{Th1} is an extension of some previously known results:
\begin{itemize}
\item When $$k(t)=g_{1-\alpha}(t)=\frac{t^{-\alpha}}{\Gamma(1-\alpha)},\alpha\in(0,1),\,\,\mathcal{L}=-\Delta,\,\,f(u)=u(u-1),$$ then it coincides with the results of Ahmad et al. \cite{AAAKT15};
\item When $$k(t)=g_{1-\alpha}(t),\alpha\in(0,1),\,\,\mathcal{L}=(-\Delta)^s,\,s\in(0,1),\,\,f(u)=u(u-1),$$ then it coincides with the results of Alsaedi et al. \cite{Kir2};
\item When $k(t)=g_{1-\alpha}(t),\alpha\in(0,1),\,\,$ then it coincides with the results of Gal and Warma \cite{Gal}.
\end{itemize}
\end{remark}
\begin{remark}\label{rem2} Part (ii) of Theorem \ref{Th1} gives a partial answer to an open question by Luchko and Yamamoto \cite[Section 5]{Luchko1} regarding the asymptotic behavior of the solution of the diffusion equation with general $\partial_t (k\ast \cdot)$.
Also, part (ii) of Theorem \ref{Th1} extends some previously known or unknown results. In particular:
\begin{itemize}
\item Let $k(t)=g_{1-\alpha}(t)$ and $l(t)=g_\alpha(t),$ then we have logarithmic decay (see \cite{AAAKT15, Kir2, Zacher1})
$$\|u(t,\cdot)\|_{L^2(\Omega)}\leq \frac{1}{1+C_{\mathcal{L}}t^\alpha}\|u_0\|_{L^2(\Omega)},\,\,\,t\geq 0;$$
\item Let $k(t)=\sum\limits_{j=1}^mg_{1-\alpha_j}(t),\,0<\alpha_m<...<\alpha_2<\alpha_1<1,$ then we have (see \cite{Yam2, Zacher1})
$$\|u(t,\cdot)\|_{L^2(\Omega)}\leq \frac{C}{1+t^{\alpha_m}}\|u_0\|_{L^2(\Omega)},\,\,C>0,\,t\geq T_1>0;$$
\item Let $k(t)=\int\limits_0^1g_\alpha(t)d\alpha,\, l(t)=\int\limits_0^\infty \frac{e^{-st}}{1+s}ds,\,t>0,$ then we have (see \cite{Luch1, Zacher1})
$$\|u(t,\cdot)\|_{L^2(\Omega)}\leq \frac{1}{1+C_{\mathcal{L}}\log t}\|u_0\|_{L^2(\Omega)},\,\,\,t\geq T_1>0;$$
\item Let $k(t)=g_{1-\alpha}(t)e^{-\mu t},\, l(t)=g_\alpha(t)e^{-\mu t}+\mu\int\limits_0^tg_\alpha(s)e^{-\mu s}ds,\,t>0,$ then we have exponential decay (see \cite{Zacher1})
$$\|u(t,\cdot)\|_{L^2(\Omega)}\leq e^{-Ct}\|u_0\|_{L^2(\Omega)},\,\,\,t\geq 0,\,C>0.$$
\end{itemize}
Note that all of the above examples belong to the class $l\not\in L^1(\mathbb{R}_+)$. When $l\in L^1(\mathbb{R}_+)$, then the solution does not decaying, but will be bounded $\|u(t,\cdot)\|_{L^2(\Omega)}\leq C,\,C>0,\, t\geq 0.$
\end{remark}
\begin{remark}\label{rem3}
If $k(t)=g_{1-\alpha}(t),\alpha\in(0,1),\,\,$ then the part (iii) of Theorem \ref{Th1} gives a partial answer to an open question posed by Gal and Warma (see \cite[Chapter 5, Problem 3(c)]{Gal}) regarding the finite time blow-up of the solution of problem \eqref{1},\eqref{2},\eqref{3}. When $\mathcal{L}=-\Delta,$ the above results coincides with the results of Vergara and Zacher \cite{Zacher1}. In the case when $\mathcal{L}=-\Delta$ \cite{AAAKT15} or $\mathcal{L}=(-\Delta)^s, s\in(0,1)$ \cite{Kir2} and $$k(t)=g_{1-\alpha}(t),\alpha\in(0,1),\,\,\,f(u)=u(u-1),$$ it is possible to obtain upper and lower estimate of the blow-up time in the form
\begin{equation*}\left(\frac{\Gamma(\alpha+1)}{4(c_0+1/2)}\right)^{\frac{1}{\alpha}}\leq T^*\leq \left(\frac{\Gamma(\alpha+1)}{c_0}\right)^{\frac{1}{\alpha}},\end{equation*} where $c_0=\int\limits_\Omega u_0(x)\phi_1(x)dx,$ $\phi_1>0$ is the first eigenfunction of Dirichlet-Laplacian.
\end{remark}

\section{Proof of main results} \label{sec3}

For the proof of our result in this section, we will make essential use of the following inequality proved in \cite{Zacher1}
\begin{equation}\label{Lp}
\|u(t)\|_{L^2(\Omega)}\partial_{t}\left(k\ast\left(\|u(\cdot)\|_{L^2(\Omega)}-\|u_0\|_{L^2(\Omega)}\right)\right)(t) \leq\int\limits_{\Omega}u\partial_{t}\left(k\ast\left(u-u_0\right)\right)(t)dx.
\end{equation}

\begin{proof}[Proof of Theorem \ref{Th2}] The existence of a unique strong local solution follows from \cite{Ke}. In \cite{Ke}, was studied an abstract nonlocal evolution equation in the form
\begin{equation} \label{aee}\partial_t (k\ast(u-u_0))+Au=f(u),\,\,\,\,t>0,\end{equation}
where $A$ is an unbounded linear operator on separable Hilbert space $H.$ In particular, for locally Lipschitz $f(u)$, the existence of a unique strong solution to problem \eqref{aee} was proved.

Let $0\leq u_0(x)\leq 1,\,x\in\bar\Omega,$ we will prove $0\leq u(x,t)\leq 1$ for $(x,t)\in \bar\Omega\times[0,T^*].$ Firstly, we show that $u\ge 0.$ Let $$\tilde{u}:=\min \left( u,0 \right),$$ hence $$\tilde{u}_0:=\min \left(u_0,0 \right)=0.$$ Multiplying both sides of equation \eqref{1} by $\tilde{u}$ and integrating over $\Omega$, we obtain
\begin{align*}\int\limits_{\Omega}\partial_t (k\ast(\tilde{u}))\tilde{u}dx+\int\limits_{\Omega}\mathcal{L}_x[\tilde{u}]\tilde{u}dx =\int\limits_{\Omega}\tilde{u}f\left(\tilde{u}\right)dx.\end{align*}
Using the hypothesis (B) and inequality \eqref{Lp} to the left hand side of last equality we have
\begin{align*}\|\tilde{u}(t)\|_{L^2(\Omega)}\partial_{t}\left(k\ast\left(\|\tilde{u}(\cdot)\|_{L^2(\Omega)}\right)\right)(t)+C_\mathcal{L}\|\tilde{u}(t)\|^2_{L^2(\Omega)} \leq\int\limits_{\Omega}\tilde{u}f\left(\tilde{u}\right)dx.\end{align*}
Since $f(v)\geq 0$ for $v\leq 0$ by hypothesis (C), we have
\begin{align*}\partial_{t}\left(k\ast v\right)(t)+C_\mathcal{L}v(t) \leq 0,\end{align*} where $v(t)=\|\tilde{u}(t)\|_{L^2(\Omega)}.$
It is obvious that $0$ is a super solution of the last inequality, then the comparison principle (see \cite[Lemma 2.5]{Zacher1}) implies
$$\|\tilde{u}(t)\|_{L^2(\Omega)}\leq 0, t\in [0,T^*],$$ hence $\tilde{u}\equiv 0$ for all $(x,t)\in\bar\Omega\times[0,T^*).$ Consequently, $u\geq 0$ for $(x,t)\in \bar\Omega\times[0,T^*].$

Now we show that $u\le 1.$ Let us denote $\hat{u}:=\min \left( 1-u,0 \right),$ hence $\hat{u}_0:=\min \left(u_0,0 \right)=0.$ Multiplying both sides of equation \eqref{1} by $\hat{u}$ and integrating over $\Omega$, one obtain
\begin{equation}\label{7}\int\limits_{\Omega}\partial_t (k\ast(\hat{u}))\hat{u}dx+\int\limits_{\Omega}\mathcal{L}_x[\hat{u}]\hat{u}dx =-\int\limits_{\Omega}\hat{u}f\left(\hat{u}\right)dx.\end{equation}
By hypothesis (C) we know that $f(v)>0$ for $v<0,$ then $-\hat{u}f\left(\hat{u}\right)\geq 0.$
It follows from the mean value theorem that $$f\left(\hat{u}\right)=f'\left(\hat{u}^*\right)\hat{u},$$ where $\hat{u}\leq \hat{u}^*\leq 0$ and  $f'\left(\hat{u}^*\right)\leq 0.$
Applying the hypothesis (B) and inequality \eqref{Lp} to \eqref{7} we arrive at
\begin{align*}\|\hat{u}(t)\|_{L^2(\Omega)}\partial_{t}\left(k\ast\left(\|\hat{u}(\cdot)\|_{L^2(\Omega)}\right)\right)(t)+C_\mathcal{L}\|\hat{u}(t)\|^2_{L^2(\Omega)} \leq C_1\|\hat{u}(t)\|^2_{L^2(\Omega)},\,C_1>0,\end{align*} hence
\begin{align*}\partial_{t}\left(k\ast w\right)(t) \leq C_1w(t),\end{align*}
where $w(t)=\|\hat{u}(t)\|_{L^2(\Omega)}$ with $w(0)=0.$
It is easy to see that 0 is a supersolution of last inequality, then
$$\|\hat{u}(t)\|_{L^2(\Omega)}\leq 0,$$ thanks to the comparison principle.
Then $\hat{u}\left(x,t\right)\equiv 0$ for all $(x,t)\in \bar\Omega\times[0,T^*]$, which implies that $u\leq 1.$

Finally we have $0\le u\le 1$ for all $(x,t)\in \bar\Omega\times[0,T^*].$ \end{proof}

\begin{proof}[Proof of Theorem \ref{Th1}] \textbf{(i)}. Let $0\leq u_0(x)\leq 1,\,x\in\bar\Omega.$
Then from Theorem \ref{Th2} it is known that there exists unique local strong solution $u$ of \eqref{1},\eqref{2},\eqref{3}, such that $0\le u\le 1$ for all $(x,t)\in \bar\Omega\times[0,T^*].$

Let us consider the problem \eqref{1},\eqref{3} on $(x,t)\in \Omega\times(T^*,T_1^*],\,T^*<T_1^*,$ with initial data \begin{equation}\label{2*}u(x,T^*)=u_1(x),\,\,\,\,\,x\in\Omega,\end{equation} where $u_1(x)$ is the value of function $u(x,t)$ at the point $T^*.$ It is known that $0\leq u_1(x)\leq 1$ for all $x\in\bar\Omega.$ Then, repeating the process of proving Theorem \ref{Th2}, we make sure that $0\le u\le 1$ for all $(x,t)\in \bar\Omega\times[T^*,T^*_1].$ Repeating this procedure enough times, we arrive at $0\le u\le 1$ for all $(x,t)\in \bar\Omega\times[0,\infty).$

\textbf{(ii).} Let $0\leq u_0(x)\leq 1,\,x\in\bar\Omega.$ Multiplying both sides of equation \eqref{1} by $u$ and integrating over $\Omega$, we obtain
\begin{align*}\int\limits_{\Omega}\partial_t (k\ast({u}-u_0)){u}dx+\int\limits_{\Omega}\mathcal{L}_x[{u}]{u}dx =\int\limits_{\Omega}{u}f\left({u}\right)dx.\end{align*}
Applying the hypothesis (B) and inequality \eqref{Lp} to the left hand side of last expression we have
\begin{align*}\|{u}(t)\|_{L^2(\Omega)}\partial_{t}\left(k\ast\left(\|{u}(\cdot)\|_{L^2(\Omega)}-\|{u}_0\|_{L^2(\Omega)}\right)\right)(t)+C_\mathcal{L}\|{u}(t)\|^2_{L^2(\Omega)} \leq\int\limits_{\Omega}{u}f\left({u}\right)dx.\end{align*}
As $0\leq u_0(x)\leq 1,\,x\in\bar\Omega,$ we have $0\le u\le 1$ for all $(x,t)\in \bar\Omega\times[0,\infty).$ Hence $f(u)\leq 0$ by hypothesis (C), then
\begin{align*}\partial_{t}\left(k\ast (U-1)\right)(t)+C_\mathcal{L}U(t) \leq 0,\end{align*} where $U(t)=\frac{1}{\|{u}_0\|_{L^2(\Omega)}}\|{u}(t)\|_{L^2(\Omega)}$ and $U(0)=\frac{1}{\|{u}_0\|_{L^2(\Omega)}}\|{u}_0\|_{L^2(\Omega)}=1.$

Let us consider the equation
\begin{align*}\partial_{t}\left(k\ast (W-1)\right)(t)+C_\mathcal{L}W(t) = 0,\end{align*} with $W(0)=1.$ Note that the last equation is equivalent to the following Volterra equation (see \cite[Page 213]{Zacher1})
\begin{align*}W(t)+C_\mathcal{L}(l\ast W)(t) = 1.\end{align*}
It is known that the solution of the last integral equation nonincreasing, in addition, it satisfies
\begin{align*}W(t) \leq \frac{1}{1+C_\mathcal{L}(l\ast 1)(t)},\,t\geq 0.\end{align*}
Since $U(t)\leq W(t),\,t\geq 0$ according to the comparison principle, then
\begin{align*}U(t) \leq \frac{1}{1+C_\mathcal{L}(l\ast 1)(t)},\,t\geq 0.\end{align*}
Finally, it gives
\begin{align*}\|{u}(t)\|_{L^2(\Omega)} \leq \frac{\|{u}_0\|_{L^2(\Omega)}}{1+C_\mathcal{L}(l\ast 1)(t)},\,t\geq 0.\end{align*}
If $l\not\in L^1(\mathbb{R}_+)$ in the last expression, then
\begin{align*}\lim\limits_{t\rightarrow+\infty}\|{u}(t)\|_{L^2(\Omega)} \leq \lim\limits_{t\rightarrow+\infty}\frac{\|{u}_0\|_{L^2(\Omega)}}{1+C_\mathcal{L}(l\ast 1)(t)}\rightarrow 0.\end{align*}

\textbf{(iii).} Assume that $u_0\geq 0,\,u_0\not\equiv 0,$ then from the Theorem \ref{Th2} it is known that there exists a local in time strong solution. Multiplying both sides of equation \eqref{1} by $\phi_1(x)>0$ and integrating over $\Omega$ one obtain
\begin{align*}\int\limits_{\Omega}\partial_t (k\ast({u}-u_0))\phi_1(x)dx+\int\limits_{\Omega}\mathcal{L}_x[{u}]\phi_1(x)dx =\int\limits_{\Omega}f\left({u}\right)\phi_1(x)dx.\end{align*}
Integration by parts of the second term of the last expression gives us
\begin{align*}\int\limits_{\Omega}\partial_t (k\ast({u}-u_0))\phi_1(x)dx+\lambda_1\int\limits_{\Omega}{u}\phi_1(x)dx =\int\limits_{\Omega}f\left({u}\right)\phi_1(x)dx.\end{align*}
Since $f(y)$ is a convex function for $y\in\mathbb{R}_+$, by virtue of the Jensen's inequality we have
$$\int\limits_{\Omega}f\left({u}\right)\phi_1(x)dx\geq f\left(\int\limits_{\Omega}u\phi_1(x)dx\right),$$ then
\begin{equation}\label{8}\partial_t (k\ast({\Phi}-\Phi_0))(t)+\lambda_1\Phi(t) \geq f(\Phi(t)),\end{equation}
where $\Phi(t)=\int\limits_{\Omega}{u}\phi_1(x)dx$ and $\Phi(0)\equiv \Phi_0=\int\limits_{\Omega}{u}_0(x)\phi_1(x)dx.$

In \cite{Zacher2}, Vergara and Zacher proved that if conditions \eqref{7**} and \eqref{7***} are satisfied, then the solution to inequality \eqref{8} blows up in finite time. Therefore, this means that the solution to problem \eqref{1},\eqref{2},\eqref{3} also blows up in a finite time.

The proof is complete.
\end{proof}

\section{Discussions about quasilinear extensions}
In this section, we discuss the possibility of extending the main results for \eqref{1},\eqref{2},\eqref{3} when $\mathcal{L}$ is a nonlinear operator $\mathcal{N}$ satisfying the following property
\begin{description}
\item[(B*)] Suppose that $\mathcal{N}_x$ is an unbounded operator with boundary condition \eqref{3} and let there exist $\gamma\in(0,\infty),$ $q>1,$ $C_\mathcal{N}>0$ such that if $u$ is the positive solution of \eqref{1}-\eqref{3}, then
\begin{equation}\label{NLineq}\int\limits_\Omega u^{q-1}(x,t)\mathcal{N}_x[u](x,t)dx\geq C_\mathcal{N}\|u(\cdot,t)\|^{q-1+\gamma}_{L^q(\Omega)},\,t>0.\end{equation}
\end{description}
Hypothesis (B*) have been considered in \cite{Vald1, Vald2} for the study of more general diffusion equations with time-fractional derivatives. As mentioned in \cite{Vald1, Vald2}, typical examples of operators satisfying (B*) are the following:
\begin{itemize}
  \item Laplacian $-\Delta$;
  \item fractional Laplacian $(-\Delta)^s$;
  \item p-Laplacian $\Delta_p=div\left(|\nabla \cdot|^{p-2}\nabla \cdot\right)$;
  \item fractional p-Laplacian;
  \item porous medium operator $-\Delta (\cdot)^m,\,m>1$;
  \item doubly nonlinear operator $div\left(|\nabla (\cdot)^m|^{p-2}\nabla (\cdot)^m\right)$;
  \item mean curvature operator $div\left(\frac{\nabla \cdot}{1+|\nabla \cdot|^2}\right)$;
  \item sum of different space-fractional operators $\sum\limits_{j=1}^m(-\Delta)^{s_j}$;
\end{itemize} and others.

Let us consider the equation
\begin{equation}\label{1*}\partial_t (k\ast(u-u_0))+\mathcal{N}_x [u]=f(u),\,\,\,\, x\in\Omega\subset\mathbb{R}^n, t>0,\end{equation}
with initial-boundary conditions \eqref{2},\eqref{3}.

In this section, we will discuss about the existence of weak solutions to problem \eqref{1*}, \eqref{2}, \eqref{3}.

{\bf Local existence.} Note that the existence of a weak solution on $[0,T]\times\Omega,$ $T<\infty$ were obtained in \cite{Zacher} for more general case of problem \eqref{1*}, \eqref{2}, \eqref{3}. It follows from this that the problem \eqref{1*}, \eqref{2}, \eqref{3} has a local weak solution on $[0,T]\times\Omega,$ for finite $T.$

Repeating the technique of proving Theorem \ref{Th1} and taking into account hypotheses (A), (B*), (C) it is easy to prove that the weak local solution of problem \eqref{1*}, \eqref{2}, \eqref{3} satisfies $$0\leq u(x,t)\leq 1\,\,\,\text{for} \,\,(x,t)\in \bar\Omega\times[0,T],$$ for $0\leq u_0(x)\leq 1,\,x\in\bar\Omega$.

{\bf Global existence.} Taking into account the existence of local weak solutions and using the method of proving Theorem \ref{Th1}, it is easy to show that for $0\leq u_0(x)\leq 1,\,x\in\bar\Omega$ there is a global in time weak solutions to problem \eqref{1*}, \eqref{2}, \eqref{3}, and it satisfies the estimate
$$0\leq u(x,t)\leq 1$$ for all $(x,t)\in \bar\Omega\times\mathbb{R}_+.$

{\bf Decay estimates.} In the general case, it is not easy to obtain decay estimate of solution to problem \eqref{1*}, \eqref{2}, \eqref{3}. Let us illustrate this with the following example. Let $k(t)=g_{1-\alpha}(t)=\frac{t^{-\alpha}}{\Gamma(1-\alpha)},\alpha\in(0,1),$ then the equation \eqref{1*} takes the form
\begin{equation}\label{1**}\partial_t (g_{1-\alpha}\ast(u-u_0))(t)+\mathcal{N}_x [u]=f(u),\,\,\,\, x\in\Omega\subset\mathbb{R}^n, t>0.\end{equation} Assume that $0\leq u_0(x)\leq 1,\,x\in\bar\Omega,$ then $0\leq u(x,t)\leq 1$ for all $(x,t)\in \bar\Omega\times\mathbb{R}_+.$ Multiplying equation \eqref{1*} by $u^{q-1}$ and integrating over $\Omega$ we have
\begin{align*}\int\limits_{\Omega}\partial_t (g_{1-\alpha}\ast({u}-u_0)){u^{q-1}}dx+\int\limits_{\Omega}\mathcal{N}_x[{u}]{u^{q-1}}dx =\int\limits_{\Omega}{u^{q-1}}f\left({u}\right)dx.\end{align*} Since $0\leq u(x,t)\leq 1,$ it follows from (C) that $${u^{q-1}}f\left({u}\right)\leq 0.$$ Then applying (B*) and following inequality (see \cite{Zacher1})
\begin{equation}\begin{split}\label{NLp}
\|u(t)\|^{q-1}_{L^q(\Omega)}&\partial_{t}\left(k\ast\left(\|u(\cdot)\|_{L^q(\Omega)}-\|u_0\|_{L^q(\Omega)}\right)\right)(t) \\&\leq\int\limits_{\Omega}u^{q-1}\partial_{t}\left(k\ast\left(u-u_0\right)\right)(t)dx.
\end{split}\end{equation}
we have that
\begin{align*}0&\geq \|u(t)\|^{q-1}_{L^q(\Omega)}\partial_{t}\left(g_{1-\alpha}\ast\left(\|u(\cdot)\|_{L^q(\Omega)}-\|u_0\|_{L^q(\Omega)}\right)\right)(t)+ C_\mathcal{N}\|u(\cdot,t)\|^{q-1+\gamma}_{L^q(\Omega)} \\&= \|u(t)\|^{q-1}_{L^q(\Omega)}\left[\partial_{t}\left(g_{1-\alpha}\ast\left(\|u(\cdot)\|_{L^q(\Omega)}-\|u_0\|_{L^q(\Omega)}\right)\right)(t)+ C_\mathcal{N}\|u(\cdot,t)\|^{\gamma}_{L^q(\Omega)}\right],\end{align*} hence
\begin{align*}\partial_{t}\left(g_{1-\alpha}\ast\left(U-U_0\right)\right)(t)+ C_\mathcal{N}U^{\gamma}(t)\leq 0,\end{align*}
where $U(t)=\|u(\cdot,t)\|_{L^q(\Omega)}.$ It is known that (see \cite{Vald2, Zacher1})
$U(t)\leq \frac{C}{1+t^{\alpha/\gamma}},$ hence
$$\|u(\cdot,t)\|_{L^q(\Omega)}\leq \frac{C}{1+t^{\alpha/\gamma}},$$
where $C>0,$ possibly depending on $C_\mathcal{N},$ $\alpha,$ $\gamma$ and $\|u_0\|_{L^q(\Omega)}.$

Similarly, the problem \eqref{1*}, \eqref{2}, \eqref{3} can be reduced to the integro-differential inequality
\begin{align*}\partial_{t}\left(k\ast\left(U-U_0\right)\right)(t)+ C_\mathcal{N}U^{\gamma}(t)\leq 0,\end{align*}
with $U(t)=\|u(\cdot,t)\|_{L^q(\Omega)}.$ However, for the last inequality there is not yet the decay estimates of solution, and this is still an open question. Consequently, for problem \eqref{1*}, \eqref{2}, \eqref{3} the question of the decay of the solutions remains open.

{\bf Blow-up of solutions.} It is quite difficult to obtain a blow-up solution for problem \eqref{1*}, \eqref{2}, \eqref{3} in the general case due to technical problems. At the moment, we have no idea how to prove a blow-up solution for problem \eqref{1*}, \eqref{2}, \eqref{3} even in case $k(t)=g_{1-\alpha}(t)=\frac{t^{-\alpha}}{\Gamma(1-\alpha)},\alpha\in(0,1),$ and we will leave this question open as well.

\section*{Acknowledgements}
This research has been funded by the Science Committee of the Ministry of Education and Science of the Republic of Kazakhstan (Grant No. AP09259578), by the FWO Odysseus 1 grant G.0H94.18N: Analysis and Partial Differential Equations and by the Methusalem programme of the Ghent University Special Research Fund (BOF) (Grant number 01M01021).

\section*{Conflict of interest}
There is no conflict of interest.

\end{document}